\documentclass[a4paper,12pt]{article}
\newcommand{\bq}{\begin{eqnarray}}
\newcommand{\bqn}{\begin{eqnarray*}}
\newcommand{\enq}{\end{eqnarray}}
\newcommand{\enqn}{\end{eqnarray*}}
\newcommand{\nin}{\noindent}

\usepackage[pdftex]{color, graphicx}
\usepackage{latexsym, amsmath, amsthm, amssymb,mathabx}
\usepackage[all]{xy}
\usepackage{float}
\usepackage{booktabs}
\usepackage{multirow}
\usepackage{pdflscape}
\usepackage{amsthm}
\usepackage{authblk}
\usepackage[titletoc]{appendix}
\newtheorem{thm}{Theorem}[section]

\numberwithin{equation}{section}

\begin{document}
\title{A Closed Form Approximation of Moments of New Generalization of Negative Binomial Distribution}
\author[1]{Sudip Roy \thanks{Corresponding author: sudip.roy@utsa.edu}}
\author[1]{Ram C. Tripathi \thanks{ram.tripathi@utsa.edu}}
\author[2]{N. Balakrishnan  \thanks{bala@mcmaster.ca}}
\affil[1]{Department of Management Science and Statistics, University of Texas at San Antonio, One UTSA Circle, San Antonio, Texas $78249$}
\affil[2]{Department of Mathematics and Statistics, McMaster University, $1280$ Main Street West,  Hamilton, Ontario, Canada $L8S$ $4L8$}
\renewcommand\Authands{ and }
\maketitle
\begin{center}
\bf{Abstract}
\end{center}
In this paper, we propose a closed form approximation to the mean and variance of a new generalization of negative binomial (NGNB) distribution arising from the Extended COM-Poisson (ECOMP) distribution developed by Chakraborty and Imoto (2016)(see \cite{chak16}). The NGNB is a special case of the ECOMP distribution and was named so by these authors.  This distribution is more flexible in terms of dispersion index as compared to its ordinary counterparts. It approaches to the  COM-Poisson distribution (Shmueli et al. $2005$) \cite{shm05} under suitable limiting conditions. The NGNB can also be obtained from the COM-Negative Hypergeometric distribution (Roy et al. $2019$)\cite{Sudip19} as a limiting distribution. In this paper, we present closed form approximations for the mean and variance of the NGNB distribution. These approximations can be viewed as the mean and variance of convolution of independent and identically distributed negative binomial populations. The proposed closed form approximations of the mean and variance will be helpful in building the link function for the generalized negative binomial regression model based on the NGNB distribution and other extended applications, hence resulting in enhanced applicability of this model.  \\

\section{Introduction}
  In this paper, we present the probability function (pf) of the  NGNB model  (Chakraborty and Imoto 2016)\cite{chak16} and propose  closed form approximations for its mean and variance.  The approximate expression for the mean can be used to develop a link function for the new generalized negative binomial regression model. We plan to develop and report this in a separate paper. This will enhance the applicability of the NGNB model.   The NGNB is a particular case of the ECOMP distribution when three of its parameters are equal. The ECOMP distribution has many important characteristics.  It approaches to the COM-Poisson distribution under suitable limiting conditions. The NGNB distribution is also quite versatile. Its dispersion index is flexible and hence, it can accommodate over-dispersed and under-dispersed data. It can be log-concave or log-convex based on the values of its parameters. Hence, it has the properties of  increasing failure rate (IFR) and decreasing failure rate (DFR). \\
  The paper is organized as follows. In Section 2, we provide the  probability function of the NGNB distribution including its parameters and briefly present its moments and probability generating function. We show that the NGNB model can be regarded as a member of the Kemp's family of distributions. In Section 3, we will present the closed form approximations of its mean and variance and investigate its behavior through a simulation study. We assess the closeness of the proposed approximations for the mean and variance to their true values by analyzing error, bias and mean square error  for different combinations of the parameters. In the appendix, we define the log-concavity(convexity) and use these to characterize the IFR(DFR) properties of the NGNB failure rate for various range of values of its parameters. The theoretical derivations of these characteristics are also provided.

\section{New Generalization of Negative Binomial Distribution (NGNB)}
 The negative binomial distribution arises by sampling with replacement from a population consisting of two types of
 items, say $S$ : Success and $F$: Failure, until we get $k$ number of successes, where the probability of success is $p$
 and  the probability of failure is $\textit{q} = 1-\textit{p}$ for each independent trial. Let $X$ denote the number of
 failures before getting $k$ successes. Then $X$ is said to have a negative binomial distribution with parameters
 \textit{k} and \textit{p}. Its pf is given by
\bq \label{e:NB} P(X=x)=\binom{x+k-1}{x} p^k q^x    \;\; \quad
x=0,1,2,\cdots. \enq This distribution is always over-dispersed (variance
greater than the mean) and belongs to the family of power series
distributions. The mean, variance and the probability generating function
(pgf) of the NB model are given by \bqn
E(X)&=&\frac{kq}{p},\\
Var(X)&=& \frac{kq}{p^2},\\
G(s)&=& E(s^X)= p^k{(1-qs)}^{-k}. \enqn
Following  similar notation as for
COM Poisson, we have described the NGNB distribution as a generalized form of the
negative binomial distribution given in (\ref{e:NB}) above. This is achieved  by raising the combinatorial term in the probability
 function (pf) of the negative binomial distribution to a power $\gamma$ which acts as the shape parameter and by normalizing
 the resulting expression. The NGNB random variable is denoted by $Y$ with its pf defined as
\bq \label{e:COM-NB}
P(Y=y)=\frac{{\binom{y+k-1}{y}}^\gamma q^{y}}{\sum_{j=0}^{\infty}{\binom{j+k-1}{j}}^\gamma q^{j}}\\
y=0,1,\cdots \nonumber \; \text{ and} \; \gamma \in \Re, \; k > 0,\; 0<q<1 .
\enq
We denote the normalizing constant by
\[ Z(\gamma,k,q)=\sum_{j=0}^{\infty}{\binom{j+k-1}{j}}^\gamma q^{j}. \]
 Whenever necessary, we will denote the distribution by $NGNB(\gamma,k,p)$. The NGNB model provides a more flexible alternative to the negative binomial, since the additional
  shape parameter $\gamma$ makes the new model more or less over-dispersed as compared to the negative binomial. When
   $\gamma=0$, the normalizing constant $Z(\gamma,k,q)= \frac{1}{1-q}$, and NGNB reduces to the geometric distribution
    with probability function $p(y)=  pq^y$. The NGNB model is a member of the more general class called  ECOMP studied by  Chakraborty and Imoto (2016) \cite{chak16}. The NGNB was independently studied  by (Roy 2016)\cite{Sudip16} in a Ph.D thesis where it was called COMP-NB. In this section, we investigate additional statistical properties of
    the NGNB (COMP-NB)  model.
\subsection{Miscellaneous results regarding NGNB distribution }
In this section, we discuss some important characteristics of the NGNB model. These include characterization of the monotonicity of its failure rate based on the log-convexity and log-concavity of its pf. We also show that when $\gamma$ is a positive integer, its pgf can be written in terms of the generalized hypergeometric function. This establishes the NGNB  as a member of the Kemp family of distributions and makes it easier to compute its moments.
\subsubsection{Characterization of failure rates of the NGNB}
The monotonicity of failure rate of a discrete model plays an important role in modeling of failure time of components measured in number of cycles until the failure occurs. As established in (Gupta, Gupta \& Tripathi 1997)\cite{gupta97}, the log-concavity (log-convexity) of the pf of a  distribution determines the monotonicity of its failure rate.
 In the appendix, we characterize the failure rate of the NGNB in terms of the log-concavity (log-convexity) of its pf.    We have showed in theorem $A.1$ in the appendix that pf of NGNB is log-concave for positive $\gamma$ when $k>1$ and log-convex when $\gamma<0$ and $k>1$. Hence, we have established that the NGNB has increasing (decreasing) failure rate for $\gamma$ greater (less) than $0$, when $k>1$. In Theorem $A.2$ of the Appendix, we also show that the NGNB model approaches to the COM-Poisson distribution with parameters $\lambda$ and $\gamma$ in the limit as $k \rightarrow \infty$ and $p \rightarrow 0$ such that $k p=\lambda$ remains fixed. This parallels to the similar limiting results relating the NB and the Poisson distributions.

\subsubsection{Moments and probability generating function (for $\gamma$ a positive integer)}
  In this section, we  express the pgf of the NGNB in the form of generalized hypergeometric function when $\gamma$
   is a positive integer. We use this representation to find the factorial moments and hence other moments of the distribution in this
   special case.\\
  Let us first introduce some functions and notations which will be used throughout:
\begin{itemize}
\item Pochhammer symbol or rising factorial:
\bqn
\begin{array}{lcl}
(b)_0=1,\\
(b)_k=b(b+1)\cdots(b+k-1), \quad k\ge 1;
\end{array}
\enqn
\item Generalized hypergeometric function:
\bqn
\,_{p}F_{q}\left(a_1,a_2,\cdots,a_p;b_1,b_2,\cdots,b_q;z\right)=\sum_{n=0}^{\infty}{\frac{(a_1)_n(a_2)_n\cdots(a_p)_n }{(b_1)_n(b_2)_n\cdots(b_q)_n}}\frac{z^n}{n!}.
\enqn
\end{itemize}
The following theorem gives an expression for the pgf of NGNB distribution
in terms of the generalized hypergeometric function.
  \begin{thm} Let $Y$ denote the random variable with the $NGNB(\gamma,k,p)$ distribution. When $\gamma$ is a positive integer, the pgf of $Y$ can be written in terms of the hypergeometric series as follows. \\
 \bq
G(\gamma,s)= E(s^y)= \frac{\,_{\gamma}F_{\gamma-1}\left(k,k,\cdots,k;1,1,\cdots,1;qs\right)}{\,_{\gamma}F_{\gamma-1}\left(k,k,\cdots,k;1,1,\cdots,1;q\right)}.
\enq
\end{thm}
 \begin{proof}
 Consider the pgf of the NGNB distribution
\bqn
G(\gamma,s)=E(s^y)=\sum_{y=0}^\infty \frac{{\binom{y+k-1}{y}}^\gamma {(sq)}^y}{\sum_{j=0}^{\infty}{\binom{j+k-1}{j}}^\gamma q^{j}}.
\enqn
The numerator in the above can be expressed as a generalized hypergeometric series if $\gamma$ is a positive integer as seen below:
\bqn
\begin{array}{lcl}
\sum_{y=0}^\infty {\binom{y+k-1}{y}}^\gamma {(sq)}^y\\
= \sum_{y=0}^\infty \left(\frac{(y+k-1)!}{y!(k-1)!}\right)^\gamma {(sq)}^y\\
= \sum_{y=0}^\infty \left(\frac{(k)_y(k-1)!}{(1)_y(k-1)!}\right)^\gamma {(sq)}^y\\
=\sum_{y=0}^\infty \frac{{(k)_y}^\gamma}{{(1)_y}^{\gamma-1}} \frac{{(sq)}^y}{y!}\\
=\,_{\gamma}F_{\gamma-1}\left(k,k,\cdots,k;1,1,\cdots,1;qs\right).\\
\end{array}
\enqn Then, the pgf can be written as
 \bq \label{eq: pgf of COMNB}
G(\gamma,s)=\frac{\,_{\gamma}F_{\gamma-1}\left(k,k,\cdots,k;1,1,\cdots,1;qs\right)}{\,_{\gamma}F_{\gamma-1}\left(k,k,\cdots,k;1,1,\cdots,1;q\right)}.
\enq Hence the proof.
\end{proof}
\vskip .15in \nin {\bf Mean, variance and higher order moments of NGNB:}\\
For the $NGNB(\gamma, k,p)$, when $\gamma$ is a positive integer,  the expectation E(Y)
can be obtained from (\ref{eq: pgf of COMNB}) and is expressed as \bqn
E(Y)&=& qk^\gamma \frac{\sum_{j=0}^\infty \binom{k+j}{j}^\gamma\frac{q^j}{(j+1)^{\gamma-1}}}{\sum_{j=0}^\infty \binom{k+j-1}{j}^\gamma q^j}\\
 &=& \textit{q} k^\gamma \frac{\,_{\gamma}F_{\gamma-1}\left(k+1,k+1,\cdots,k+1;2,2,\cdots,2;q\right)}{\,_{\gamma}F_{\gamma-1}\left(k,k,\cdots,k;1,1,\cdots,1;q\right)}.
\enqn
It can be seen that E(Y) exists and is finite because the hypergeometric series has a radius of convergence at
 $1$ for $|q|<1$. Both the series are nearly-poised hypergeometric series of first kind. The $r$th order factorial moment
  and hence the variance of NGNB can be obtained from (\ref{eq: pgf of COMNB}) by evaluating its $r$th derivative
  at $s=1$ when $\gamma$ is a positive integer. As seen above, the mean, variance and higher order moments can not be evaluated in closed form. It would be important if some approximations can
  be developed for the mean and variance of the NGNB model in some special  cases. This will be addressed in Section 3 below.\\

 \subsubsection{The NGNB as a Member of Kemp's Family of Distributions When $\gamma$ is a Positive Integer}
The broad family of generalized hypergeometric probability distributions in the Kemp family \cite{gur77} is known for many useful properties. Most common distributions which are members of Kemp family are Poisson, binomial, negative binomial, hypergeometric and negative hypergeometric distributions. The distributions in this family have their pgf as
\bqn
G_1(z)=\frac{\,_{p}F_{q}\left((a);(b);\lambda z\right)}{\,_{p}F_{q}\left((a);(b);\lambda\right)},
\enqn
where $\,_{p}F_{q}\left((a);(b);z\right)$ is the generalized hypergeometric series with $p$ numerator parameters $(a_1,a_2,\cdots,a_p)$, and $q$ denominator parameters $(b_1,b_2,\cdots,b_q)$.
For the Kemp family of distributions the ratio of successive probabilities is given by
\bq \frac{p(y+1)}{p(y)}=\frac{(a_1+y)(a_2+y)\cdots (a_p+y)}{(b_1+y)(b_2+y)\cdots (b_q+y)}\frac{\lambda}{1+y}.\enq
\begin{thm} The ratio of successive probabilities for the NGNB distribution can be written in the form of the ratio of the probabilities of Kemp Type $1A(i)$ families of distributions.
\end{thm}
\begin{proof}
The ratio of successive probabilities for the NGNB distribution can be written in the form of the Kemp Type $1A(i)$ families of distributions as follows :
\bqn
\frac{p(y+1)}{p(y)}&=&\frac{\binom{y+1+k-1}{y+1}^\gamma q^{y+1}}{\binom{y+k-1}{y}^\gamma q^{y}}\\
&=&\left(\frac{y+k}{y+1}\right)^\gamma  q\\
&=& \frac{(k+y)^\gamma}{(1+y)^{\gamma-1}}\frac{q}{1+y},
\enqn
 which has $\gamma$ numerator factors, $((k+y),(k+y),\cdots,(k+y))$, and $\gamma-1$ denominator factors $((1+y),(1+y),\cdots,(1+y))$ (See equation (2.5)). \\
 Hence the proof.\\
 \end{proof}
 In the following section, we develop approximate expressions for the mean and variance of the NGNB model.

 \section{Closed-form approximate expressions for the mean and variance of the NGNB distribution}
There is no closed form approximations for the mean and variance of the NGNB distribution. It would be of interest to  develop approximate expressions for the mean and variance which may help in building the link function for NGNB regression models. With this in mind, we develop approximate expressions for the mean and variance of the NGNB model and investigate the accuracy of these expressions.
We also  compare the values of the approximate expressions  with their exact values for certain range of
 its parameters and assess the approximations. These approximations are developed based on observing patterns in the tabulated values of the mean and variance for a range of parameter values.
 In
  Tables \ref{tab: COMNB moments1}-\ref{tab: COMNB moments3}, we present values of its mean
 and variance obtained by numerical calculations for a range of values of its parameters. After a close
  examination of these values, especially for large $\gamma$, a pattern seems to emerge to approximate its mean and variance.
  Based on our observations, we propose the following approximations:\\
\[ E(Y) \approx \frac{k \gamma q}{1-q},\ \text{  and  }  Var(Y)\approx\frac{k\gamma q}{(1-q)^2}.\] \\ The closed form approximation of the mean of NGNB has a relation with the mean of convolution of independent and identically distributed negative binomial populations. Thus, the approximate mean and variance of $Y$ can be interpreted as the mean and variance of  $\gamma$ independent and identically distributed populations following the negative binomial distributions referred in (\ref{e:NB}) with mean and variance respectively \[ E(W)= \frac{kq}{1-q},\ \text{  and  }  Var(W)= \frac{kq}{(1-q)^2},\] where $W$ is the corresponding NB random variable with the pdf (2.1).
In the paper from Chakraborty and Imoto 2016 \cite{chak16}, they have also discussed the asymptotic mean in section 2.8 for ECOMP $(\gamma, p, \alpha, \beta)$ which is given below
\bq \label{e: asym ECOMP}
{p}^{\frac{1}{\alpha-\beta}}+ \frac{1-\alpha +(2\gamma-1)\beta}{2(\alpha-\beta)}
\enq
As discussed in their paper, the NGNB distribution is derived from ECOMP when three of its parameters are same, i.e., $\alpha=\beta=\gamma$. However, under this condition, the equation for the ECOMP mean in equation (\ref{e: asym ECOMP}) does not exist for NGNB distribution. Thus, our closed form approximation developed here  extends the properties of NGNB distribution.\\
We now compare the exact values of the mean and variance of NGNB with their proposed approximations for a range of its
parameter values. We also present an analysis of the errors committed when
approximating their exact  values by the above proposed expressions.
 \vskip .15in
 \nin
 {\bf Analysis of errors in approximating $E(Y)$ and $Var(Y)$}\\
 Tables 1, 2 and 3 provide exact and approximate values of $E(Y)$ and
 $Var(Y)$ for a range of values of $\gamma, q $  and $k$. From
 Tables 1 and 2, it can be seen  that when $\gamma \in [0.3, 0.9]$, $q \in [ .1,.8]$ and $k$ takes integer values from $5$ to $9$, the average
 value of the error $E(Y)-k \gamma q/(1-q)$ is $-.0183$ with mean squared
 error (MSE)$=0.3645$. From Table 3, it can be seen  that when $\gamma > 1$ with its range in $[1.2, 2]$,  $q \in [ .4,.8]$ and
  $k$ takes integer values from $5$ to $9$, the average
 value of the error $E(Y)-k \gamma q/(1-q)$ is $0.2511$ with the $MSE=0.3444$.\\
 Similarly, from Tables 1 and 2, it can be seen  that when $\gamma \in [0.3, 0.9]$, $q \in [ .1,.8]$ and $k$ takes integer
 values from $5$ to $9$, the average
 value of the error $Var(Y)-k \gamma q/(1-q)^2$ is $-0.0262$ with $MSE=0.0377$. From Table 3, it can be seen  that when $\gamma > 1$ with its
values in $[1.2, 2]$,  $q \in [ .4,.8]$ and
  $k$ takes integer values from $5$ to $9$, the average
 value of the error $Var(Y)-k \gamma q/(1-q)^2$ is $-0.1438$ with the $MSE=0.1575$.\\

On comparing the average and the MSE of the errors in approximating the mean,
we observe that for $\gamma < 1$, the approximate expression overestimates
the true value of the mean. For $\gamma > 1$, it underestimates the true
value, the approximation being closer when $\gamma < 1$. The values of
the MSE remain relatively similar for the two ranges of
$\gamma$ considered.\\

 On comparing the average and the MSE of the errors in approximating the
 variance, we observe that  the approximate expression overestimates
the true variance  for all the values of $\gamma$ considered here. The
approximate expression is closer to the true value of the variance for
$\gamma < 1$ as compared to the case when $\gamma > 1$. The value of  the MSE
is smaller for the case when $\gamma < 1$. Based on this analysis,  we
conclude that the proposed approximate expression for the variance works
slightly better than the corresponding expression for approximating the mean.

\begin{center}

\begin{tabular}{|c|} \hline

Put Table \ref{tab: COMNB moments1}, Table \ref{tab: COMNB moments2},  Table \ref{tab: COMNB moments3}  here \\ \hline

\end{tabular}
\end{center}
\section{Concluding remarks}
In this paper, we have presented a closed form approximation of the  NGNB distribution. In particular, we have shown that the NGNB approaches to the COM-Poisson distribution under certain conditions. We have investigated some characteristics of the NGNB model, such as, its probability generating function, moments, and limiting properties. We have also investigated the log-concavity and log-convexity of the pf of this distribution and used them to characterize the monotonicity of its failure rate. In a future paper, we plan to formulate a generalized regression model based on the proposed approximate form of the mean of the NGNB model.
\\ \\
\begin{appendices}
\section{APPENDIX}

\subsection{Log-concavity and log-convexity of the NGNB distribution}
The failure rate function for a discrete random variable $Y$  with pf $p(y)$ is defined as
 $$ r(y)=\frac{P(Y=y)}{P(Y\ge y)}.$$
We characterize the failure rate of the NGNB model in terms of log-convexity and log-concavity of its pf which we investigate below.\\
We use the ratio of two consecutive probabilities to determine the log-concavity and log-convexity of a distribution. This is used for determining the monotonicity of failure rates. We will use the following results in this regard from (Kemp 2005) \cite{kemo05}, page $217$: A distribution with pf $p(y)$ is
\bq
 \text{ log-convex, if  }\  \frac{p(y) p(y+2)}{p(y+1)^2} &>& 1  \\
\text{ and log-concave, if  }\  \frac{p(y) p(y+2)}{p(y+1)^2}& <& 1. \enq
The following theorem characterizes the failure rate  of the NGNB in terms of its log-convexity and log-concavity for various values of $\gamma$:\\
\begin{thm} The NGNB has increasing (decreasing) failure rate for $\gamma$ greater (less) than $0$, when $k>1$. \end{thm}
\begin{proof}
For the NGNB model, we have,
\bqn
\frac{p(y+2)}{p(y+1)}=\frac{{\binom{y+2+k-1}{y+2}}^\gamma q^{y+2}}{{\binom{y+1+k-1}{y+1}}^\gamma q^{y+1}}\\
={\left(\frac{y+k+1}{y+2}\right)}^\gamma q\\
\frac{p(y+1)}{p(y)}=\frac{{\binom{y+1+k-1}{y+1}}^\gamma q^{y+1}}{{\binom{y+k-1}{y}}^\gamma q^{y}}\\
={\left(\frac{y+k}{y+1}\right)}^\gamma q.\\
\enqn Hence
\bq \label{e:ratio}
\frac{p(y)
p(y+2)}{p(y+1)^2}={\left(\frac{(y+k+1)(y+1)}{(y+k)(y+2)}\right)}^\gamma. \enq
For the expression (\ref{e:ratio}) to be less than $1$, we have \bqn
\begin{array}{lcl}
{\left((y+k+1)(y+1)\right)}^\gamma < {\left((y+k)(y+2)\right)}^\gamma\\
\Rightarrow {(y+k+1)(y+1)} < {(y+k)(y+2)} \; \text{ for}\;  \text{positive}\; \gamma\\
\Rightarrow (y+k)(y+1)+(y+1) < (y+k)(y+1)+(y+k)\\
\Rightarrow y+1<y+k\\
\Rightarrow k>1.
\end{array}
\enqn
So, the pf of NGNB is log-concave for positive $\gamma$ when $k>1$.\\ Similarly, we can prove that its pf is log-convex when $\gamma<0$ and $k>1$.\\
Hence the proof. \end{proof}
\subsection{Characterization of the failure rate of NGNB}
  When $\gamma >0$, the pf of NGNB is log-concave for $k > 1$ and  we can conclude that in this case, NBNG has an increasing failure rate (IFR) \cite{gupta97}.  When $\gamma <0$ and $k>1$, the pf of NGNB is log-convex, hence it has decreasing failure rate. The Figure \ref{fig: Nbposfailure} shows graph for failure rate curves. It shows that the NBNG has IFR for positive value of $\gamma$ and $k>1$, by increasing the value of positive $\gamma$, we are able to lower down the failure rate with respect to its occurrence. The failure rate curves of NGNB, for $\gamma <0$ and $k>1$ show that the failure rate decreases with respect to failure occurrences as $\gamma$ increases.
\begin{figure}[H]
  \centering{
  \resizebox{12cm}{!}{\includegraphics{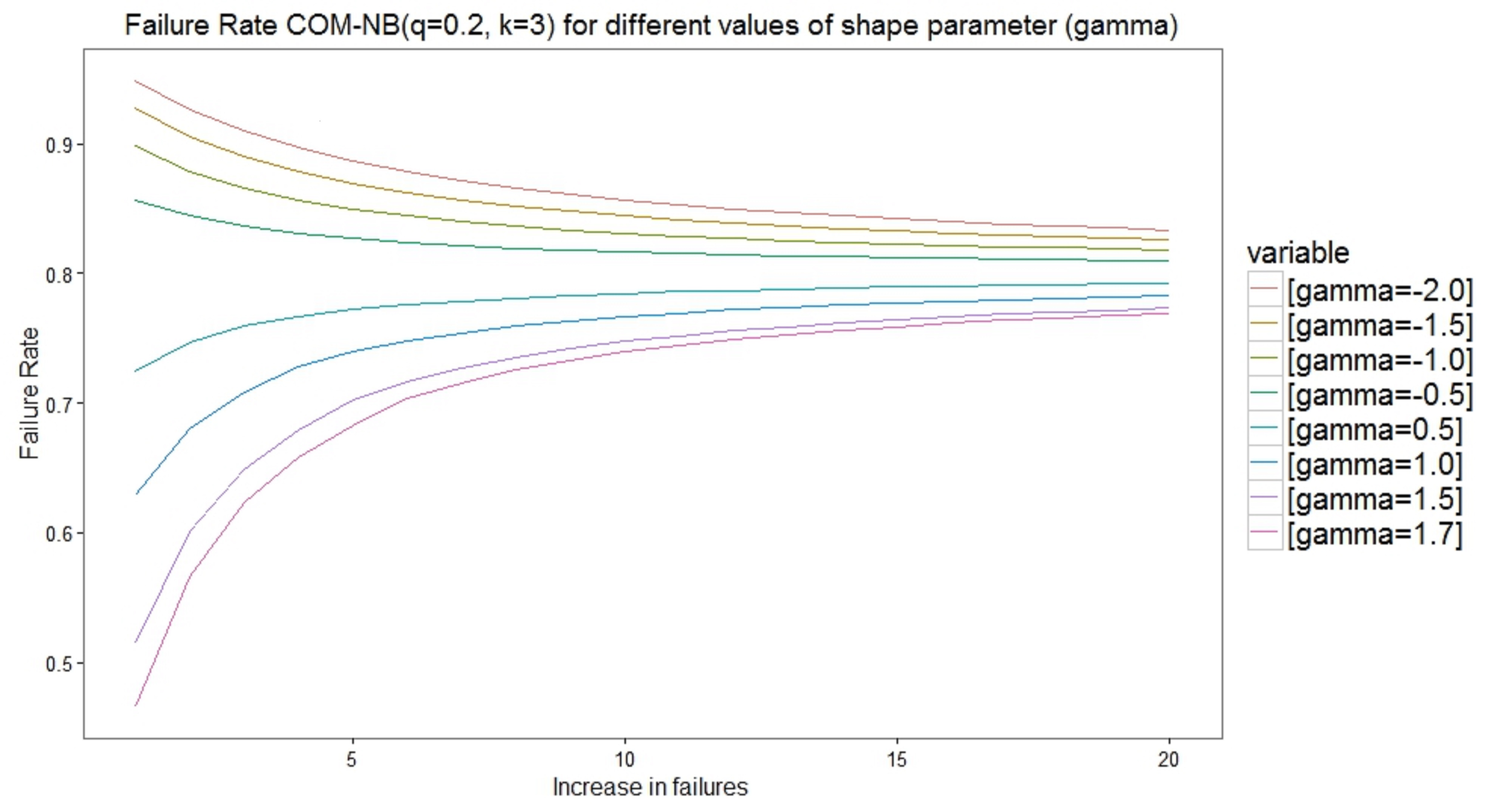}}
  }
 \caption{Failure Rate Plot for NGNB(q=0.2,k=3) for positive and negative values of $\gamma$  showing IFR and DFR Behavior respectively   \label{fig: Nbposfailure}}
\end{figure}
\subsection{Convergence of NGNB to COM-Poisson}
As is well known, negative binomial approaches the Poisson distribution under certain limiting conditions. In this section, we will show that similar relationships hold between the COM-Negative Binomial and COM-Poisson distributions. \\
\begin{thm} NBNG ($\gamma$,$k$,$q$) approaches to COM-Poisson($\lambda$) when $k\rightarrow\infty$ and $q \rightarrow 0$, such that
   $\lambda=k^\gamma q$ remains fixed.\end{thm}
\begin{proof}
For the NGNB distribution in (\ref{e:COM-NB}), the numerator can be written as
\bq
\begin{array}{lcl} \label{eq: COMP-NB convergence}
\left({\frac{(y+k-1)!}{y!(k-1)!}}\right)^\gamma {q}^y  \nonumber\\
=\frac{q^y}{y!^\gamma}\left(\frac{(y+k-1)(y+k-2)\cdots(y+k-1-(y-1))(k-1)!}{(k-1)!}\right)^\gamma\nonumber\\
=\frac{q^y (k)^{y\gamma}}{y!^\gamma}{\left((1+\frac{y-1}{k})(1+\frac{y-2}{k})\cdots(1+\frac{y-y}{k})\right)^\gamma} \nonumber\\
=\frac{{(q k^\gamma)}^y}{y!^\gamma}{\left((1+\frac{y-1}{k})(1+\frac{y-2}{k})\cdots(1+\frac{y-y}{k})\right)}^\gamma  \nonumber\\
\rightarrow \frac{\lambda^y}{(y!)^\gamma}
\end{array}
\enq
as $k\rightarrow\infty$ and $q \rightarrow 0$, such that $\lambda=k^\gamma q$ remains fixed.\\
 The same approach can be used to write the denominator in a similar form. Hence,
\bqn
\frac{{\binom{y+k-1}{y}}^\gamma q^{y}}{\sum_{j=0}^{\infty}{\binom{j+k-1}{j}}^\gamma q^{j}} \rightarrow \frac{\lambda^y}{{y!}^\gamma}\frac{1}{\sum_{j=0}^\infty \frac{\lambda^j}{{j!}^\gamma}} .
\enqn
This is the pf of the COM-Poisson distribution.\\
Hence the proof. \end{proof}
\end{appendices}

\begin{landscape}
 \begin {table}[H]
\footnotesize
\caption{Exact and approximate closed-form of mean and variance for various NGNB parameters ($\gamma$,q,k), $\gamma$ = (0.3, 0.4, 0.5, 0.6, 0.7, 0.8, 0.9), q=(0.1,0.2, 0.3, 0.4, 0.5), k=(5, 6, 7, 8, 9, 10)) \label{tab: COMNB moments1}}
\begin{center}
\scalebox{0.65}{
\begin{tabular}{|c|c|c|c|c|c|c|c|c|c|c|c|c|c|c|c|c|c|c|c|c|c|c|c|c|c|c|c|c|c|} \hline
  \multirow{1}{*}{$q$} & \multicolumn{5}{|c|}{$\gamma=0.3$}&\multicolumn{4}{|c|}{$\gamma=0.4$}&\multicolumn{4}{|c|}{$\gamma=0.5$}&\multicolumn{4}{|c|}{$\gamma=0.6$}&\multicolumn{4}{|c|}{$\gamma=0.7$} &\multicolumn{4}{|c|}{$\gamma=0.8$} &\multicolumn{4}{|c|}{$\gamma=0.9$}   \\ \hline
  \multirow{6}{*}{0.1}&k& E(Y)&$\frac{k \gamma q}{1-q}$ & Var(Y)&$\frac{k \gamma q}{(1-q)^2}$&E(Y)&$\frac{k \gamma q}{1-q}$ & Var(Y)&$\frac{k \gamma q}{(1-q)^2}$&E(Y)&$\frac{k \gamma q}{1-q}$ & Var(Y)&$\frac{k \gamma q}{(1-q)^2}$&E(Y)&$\frac{k \gamma q}{1-q}$ & Var(Y)&$\frac{k \gamma q}{(1-q)^2}$&E(Y)&$\frac{k \gamma q}{1-q}$ & Var(Y)&$\frac{k \gamma q}{(1-q)^2}$&E(Y)&$\frac{k \gamma q}{1-q}$ & Var(Y)&$\frac{k \gamma q}{(1-q)^2}$&E(Y)&$\frac{k \gamma q}{1-q}$ & Var(Y)&$\frac{k \gamma q}{(1-q)^2}$\\ \cline{2-30}
  &5&0.18& 0.17& 0.21& 0.19& 0.22& 0.22& 0.24& 0.25& 0.25& 0.28& 0.29& 0.31&0.30& 0.33& 0.34& 0.37& 0.35& 0.39& 0.40& 0.43& 0.41& 0.44& 0.46& 0.49& 0.48& 0.50& 0.54& 0.56\\ \cline{2-30}
  &6&0.19& 0.20& 0.22& 0.22&  0.23& 0.27& 0.27& 0.30& 0.28& 0.33& 0.32& 0.37& 0.34& 0.40& 0.38& 0.44& 0.40& 0.47& 0.46& 0.52& 0.48& 0.53& 0.54& 0.59 &0.57& 0.60& 0.64 &0.67\\ \cline{2-30}
  &7&0.20& 0.23& 0.23& 0.26&  0.25& 0.31& 0.29& 0.35&  0.30& 0.39& 0.35& 0.43& 0.37& 0.47& 0.43& 0.52& 0.45& 0.54& 0.52& 0.60& 0.54& 0.62& 0.62& 0.69& 0.65& 0.70& 0.74& 0.78 \\ \cline{2-30}
  &8&0.21& 0.27& 0.24& 0.30& 0.26& 0.36& 0.30& 0.40& 0.33& 0.44& 0.38& 0.49& 0.40& 0.53& 0.47& 0.59& 0.50& 0.62& 0.57& 0.69& 0.61& 0.71& 0.70& 0.79& 0.74& 0.80& 0.84& 0.89 \\ \cline{2-30}
 &9& 0.22& 0.30& 0.25& 0.33& 0.28& 0.40& 0.32& 0.44& 0.35& 0.50& 0.41& 0.56&  0.44& 0.60& 0.51& 0.67& 0.54& 0.70& 0.63& 0.78& 0.67& 0.80& 0.77& 0.89& 0.82& 0.90& 0.94& 1.00 \\ \cline{2-30}
  &10& 0.23& 0.33& 0.26& 0.37& 0.29& 0.44& 0.34& 0.49& 0.37& 0.56& 0.43& 0.62& 0.47& 0.67& 0.55& 0.74&  0.59& 0.78& 0.69& 0.86& 0.74& 0.89& 0.85& 0.99&  0.91& 1.00& 1.03& 1.11 \\ \hline
  \multirow{6}{*}{0.2}&5&0.42& 0.38& 0.54& 0.47&  0.50& 0.50& 0.64& 0.62& 0.59& 0.62& 0.76& 0.78& 0.69& 0.75& 0.90& 0.94&  0.81& 0.88& 1.05& 1.09& 0.94& 1.00& 1.21& 1.25&  1.09& 1.12& 1.38& 1.41 \\ \cline{2-30}
  &6&0.45& 0.45& 0.58& 0.56& 0.54& 0.60& 0.71& 0.75& 0.65& 0.75& 0.86& 0.94& 0.65& 0.75& 0.86 &0.94& 0.78& 0.90& 1.03 &1.12 &0.93 &1.05& 1.22 &1.31& 1.10& 1.20 &1.43 &1.50 \\ \cline{2-30}
  &7& 0.47& 0.52& 0.62& 0.66& 0.58& 0.70& 0.77& 0.87& 0.71& 0.88& 0.95& 1.09& 0.87& 1.05& 1.15& 1.31&  1.05& 1.22& 1.39& 1.53&  1.26& 1.40& 1.64& 1.75& 1.49& 1.57& 1.91& 1.97 \\ \cline{2-30}
  &8& 0.50& 0.60& 0.65& 0.75& 0.62& 0.80& 0.83& 1.00& 0.77& 1.00& 1.04& 1.25& 0.95& 1.20& 1.28& 1.50& 1.17& 1.40& 1.56& 1.75& 1.42& 1.60& 1.86& 2.00& 1.69& 1.80& 2.18& 2.25 \\ \cline{2-30}
  &9&0.52& 0.67 &0.69& 0.84& 0.66& 0.90& 0.88& 1.12 &0.83& 1.12& 1.12& 1.41& 1.04& 1.35 &1.40& 1.69& 1.28& 1.57& 1.73& 1.97& 1.57& 1.80 &2.07& 2.25& 1.89& 2.02&2.44& 2.53 \\ \cline{2-30}
  &10&0.54& 0.75& 0.72& 0.94& 0.69& 1.00& 0.94& 1.25& 0.88& 1.25& 1.21& 1.56& 1.12 &1.50& 1.53& 1.88& 1.40& 1.75& 1.89& 2.19& 1.73& 2.00& 2.29& 2.50& 2.10& 2.25& 2.71& 2.81 \\ \hline
  \multirow{6}{*}{0.3}&5&0.73& 0.64& 1.10& 0.92& 0.87& 0.86& 1.31& 1.22& 1.03& 1.07& 1.56& 1.53& 1.21& 1.29& 1.83& 1.84& 1.42& 1.50& 2.12& 2.14& 1.64& 1.71& 2.43& 2.45& 1.88& 1.93 &2.74& 2.76 \\ \cline{2-30}
   &6& 0.79& 0.77& 1.19 &1.10& 0.96& 1.03& 1.46& 1.47& 1.15& 1.29& 1.77& 1.84 & 1.38& 1.54& 2.11& 2.20& 1.64& 1.80& 2.49& 2.57& 1.93& 2.06& 2.88& 2.94& 2.24& 2.31& 3.27& 3.31 \\ \cline{2-30}
   &7&0.84& 0.90 &1.28& 1.29& 1.03& 1.20& 1.60& 1.71& 1.27& 1.50& 1.98& 2.14& 1.55& 1.80& 2.40& 2.57& 1.86& 2.10& 2.85& 3.00& 2.21& 2.40& 3.33& 3.43& 2.59& 2.70& 3.81& 3.86 \\\cline{2-30}
   &8&0.88& 1.03& 1.36& 1.47& 1.11& 1.37& 1.74& 1.96& 1.39& 1.71& 2.18& 2.45& 1.71& 2.06& 2.68& 2.94& 2.08& 2.40& 3.22& 3.43& 2.50& 2.74& 3.78& 3.92& 2.95& 3.09& 4.34& 4.41 \\ \cline{2-30}
   &9& 0.92& 1.16& 1.44& 1.65& 1.18& 1.54& 1.87& 2.20& 1.50& 1.93& 2.38& 2.76& 1.87& 2.31& 2.96& 3.31& 2.30& 2.70& 3.59& 3.86& 2.78& 3.09& 4.23& 4.41& 3.30& 3.47& 4.88& 4.96 \\ \cline{2-30}
   &10& 0.97& 1.29& 1.52& 1.84& 1.25& 1.71& 2.00& 2.45& 1.61& 2.14& 2.59& 3.06& 2.03& 2.57& 3.25& 3.67& 2.52& 3.00& 3.96& 4.29& 3.07& 3.43& 4.69& 4.90& 3.66& 3.86& 5.41& 5.51\\ \hline
 \multirow{6}{*}{0.4}&5&1.17& 1.00& 2.06& 1.67& 1.39& 1.33& 2.48& 2.22& 1.65& 1.67& 2.94& 2.78&1.94& 2.00& 3.44 &3.33& 2.25& 2.33& 3.96& 3.89& 2.59& 2.67& 4.49& 4.44& 2.96& 3.00& 5.02& 5.00 \\ \cline{2-30}
 &6&1.26& 1.20& 2.25& 2.00& 1.54& 1.60& 2.78& 2.67&1.86& 2.00& 3.37& 3.33& 2.22& 2.40& 4.00& 4.00& 2.63& 2.80 &4.66& 4.67& 3.06 &3.20 &5.34& 5.33 &3.52& 3.60 &6.01& 6.00 \\ \cline{2-30}
 &7&1.35& 1.40& 2.44& 2.33& 1.68& 1.87& 3.08& 3.11& 2.06&2.33& 3.79& 3.89& 2.51 &2.80 &4.57& 4.67& 3.00& 3.27& 5.37& 5.44 &3.52& 3.73& 6.18& 6.22 &4.08& 4.20& 6.99 &7.00\\\cline{2-30}
 &8& 1.43& 1.60& 2.62 &2.67& 1.81& 2.13& 3.37& 3.56& 2.26& 2.67& 4.22& 4.44& 2.79 &3.20& 5.14& 5.33& 3.36& 3.73& 6.09& 6.22& 3.99 &4.27 &7.04 &7.11 &4.65 &4.80& 7.97 &8.00 \\\cline{2-30}
 &9& 1.51 &1.80& 2.80& 3.00& 1.94& 2.40& 3.66& 4.00& 2.46& 3.00& 4.65& 5.00& 3.06& 3.60& 5.71& 6.00& 3.73& 4.20& 6.80& 7.00& 4.45& 4.80& 7.89& 8.00& 5.21& 5.40& 8.96& 9.00 \\\cline{2-30}
 &10& 1.58& 2.00& 2.97& 3.33& 2.07& 2.67& 3.95& 4.44& 2.66& 3.33& 5.08& 5.56&3.34& 4.00& 6.29& 6.67&4.10& 4.67& 7.52 &7.78 &4.92& 5.33& 8.74& 8.89&  5.78&  6.00&  9.94 &10.00 \\ \hline
 \multirow{6}{*}{0.5}&5&1.80 &1.50 &3.84 &3.00& 2.15& 2.00& 4.63& 4.00& 2.55& 2.50& 5.48& 5.00& 2.99& 3.00& 6.37& 6.00 &3.46& 3.50& 7.28& 7.00& 3.95& 4.00& 8.20& 8.00& 4.47& 4.50& 9.11& 9.00\\ \cline{2-30}
 &6&1.95& 1.80& 4.24& 3.60& 2.39& 2.40& 5.24& 4.80 &2.89& 3.00 &6.33& 6.00& 3.45 &3.60 &7.46& 7.20 &4.04& 4.20& 8.61& 8.40 &4.67& 4.80& 9.76& 9.60& 5.33 & 5.40 &10.89& 10.80\\\cline{2-30}
 &7&2.10& 2.10& 4.63& 4.20&2.63& 2.80& 5.85& 5.60& 3.23& 3.50& 7.17 &7.00& 3.91& 4.20& 8.56& 8.40& 4.63& 4.90& 9.95& 9.80&  5.40 & 5.60& 11.32 &11.20 & 6.19&  6.30& 12.67& 12.60\\\cline{2-30}
 &8&2.24& 2.40& 5.02& 4.80& 2.86& 3.20& 6.46& 6.40& 3.57& 4.00& 8.03& 8.00& 4.36& 4.80& 9.66& 9.60&  5.22&  5.60& 11.28& 11.20& 6.12&  6.40& 12.89& 12.80& 7.05&7.20& 14.46 & 14.40\\\cline{2-30}
 &9&2.38& 2.70& 5.40 &5.40& 3.08 &3.60& 7.07 &7.20& 3.90& 4.50 &8.89& 9.00&  4.82 & 5.40& 10.76& 10.80 &5.81 & 6.30& 12.63 &12.60 & 6.84 & 7.20 &14.45& 14.40&  7.91 & 8.10& 16.24& 16.20\\\cline{2-30}
 &10& 2.51& 3.00& 5.78 &6.00& 3.30& 4.00& 7.68& 8.00&  4.24&  5.00&  9.75& 10.00&  5.28&  6.00& 11.87& 12.00 & 6.40&  7.00& 13.97& 14.00 & 7.57&  8.00& 16.02& 16.00&  8.77&  9.00& 18.03& 18.00\\ \hline

\end{tabular}}
\end{center}
\end{table}
\end{landscape}
\newpage
\begin{landscape}
 \begin {table}[H]
\footnotesize
\caption{Exact and approximate closed-form of mean and variance for various NGNB parameters ($\gamma$,q,k), $\gamma$ = (0.3, 0.4, 0.5, 0.6, 0.7, 0.8, 0.9), q=(0.6, 0.7, 0.8, 0.9), k=(5, 6, 7, 8, 9, 10) \label{tab: COMNB moments2}}
\begin{center}
\scalebox{0.65}{
\begin{tabular}{|c|c|c|c|c|c|c|c|c|c|c|c|c|c|c|c|c|c|c|c|c|c|c|c|c|c|c|c|c|c|} \hline
  \multirow{1}{*}{$q$} & \multicolumn{5}{|c|}{$\gamma=0.3$}&\multicolumn{4}{|c|}{$\gamma=0.4$}&\multicolumn{4}{|c|}{$\gamma=0.5$}&\multicolumn{4}{|c|}{$\gamma=0.6$}&\multicolumn{4}{|c|}{$\gamma=0.7$} &\multicolumn{4}{|c|}{$\gamma=0.8$} &\multicolumn{4}{|c|}{$\gamma=0.9$}   \\ \hline
  \multirow{6}{*}{}&k& E(Y)&$\frac{k \gamma q}{1-q}$ & Var(Y)&$\frac{k \gamma q}{(1-q)^2}$&E(Y)&$\frac{k \gamma q}{1-q}$ & Var(Y)&$\frac{k \gamma q}{(1-q)^2}$&E(Y)&$\frac{k \gamma q}{1-q}$ & Var(Y)&$\frac{k \gamma q}{(1-q)^2}$&E(Y)&$\frac{k \gamma q}{1-q}$ & Var(Y)&$\frac{k \gamma q}{(1-q)^2}$&E(Y)&$\frac{k \gamma q}{1-q}$ & Var(Y)&$\frac{k \gamma q}{(1-q)^2}$&E(Y)&$\frac{k \gamma q}{1-q}$ & Var(Y)&$\frac{k \gamma q}{(1-q)^2}$&E(Y)&$\frac{k \gamma q}{1-q}$ & Var(Y)&$\frac{k \gamma q}{(1-q)^2}$\\ \cline{2-30}
  \multirow{6}{*}{0.6}&5&2.78& 2.25& 7.45& 5.62& 3.33& 3.00& 8.99& 7.50&  3.94&  3.75& 10.60&  9.37& 4.60&  4.50& 12.24& 11.25&  5.29&  5.25& 13.89& 13.12& 6.01&  6.00& 15.53& 15.00&  6.75&  6.75& 17.15& 16.87 \\ \cline{2-30}
  &6&3.04& 2.70& 8.31& 6.75& 3.73&  3.60& 10.26&  9.00&  4.51&  4.50& 12.30& 11.25&  5.34&  5.40& 14.38& 13.50&  6.22&  6.30& 16.45& 15.75&  7.13&  7.20& 18.50& 18.00& 8.06&  8.10& 20.51& 20.25\\ \cline{2-30}
  &7&3.29& 3.15& 9.16& 7.87& 4.13&  4.20& 11.53& 10.50&5.07&  5.25& 14.02& 13.12& 6.08&  6.30& 16.53& 15.75& 7.15&  7.35& 19.02& 18.37&8.25&  8.40& 21.46& 21.00& 9.37&  9.45& 23.87 &23.62\\ \cline{2-30}
&8&3.54&3.6&10&9&4.52&4.8&12.81&12&5.63&6&15.75&15&6.82&7.2&18.69&18&8.07&8.4&21.59&21&9.36&9.6&24.43&24&10.67&10.8&27.23&27\\ \cline{2-30}
&9&3.78&4.05&10.84&10.12&4.91&5.4&14.1&13.5&6.19&6.75&17.48&16.87&7.56&8.1&20.86&20.25&9&9.45&24.16&23.62&10.48&10.8&27.4&27&11.98&12.15&30.59&30.37\\ \cline{2-30}
&10&4.02&4.5&11.69&11.25&5.3&6&15.39&15&6.74&7.5&19.23&18.75&8.3&9&23.02&22.5&9.93&10.5&26.73&26.25&11.6&12&30.36&30&13.29&13.5&33.95&33.75\\ \hline
 \multirow{6}{*}{0.7}&5&4.46&  3.50& 15.99& 11.67& 5.36&  4.67& 19.23& 15.56& 6.33&  5.83& 22.55& 19.44& 7.35&  7.00& 25.87& 23.33& 8.40&8.17& 29.18& 27.22& 9.48& 9.33& 32.44& 31.11& 10.57& 10.50& 35.68& 35.00 \\ \cline{2-30}
 &6&4.93&4.2&17.97&14&6.06&5.6&22.08&18.67&7.29&7&26.27&23.33&8.58&8.4&30.44&28&9.9&9.8&34.56&32.67&11.26&11.2&38.63&37.33&12.62&12.6&42.66&42\\ \cline{2-30}
&7&5.38&4.9&19.95&16.33&6.76&6.53&24.95&21.78&8.24&8.17&30.01&27.22&9.8&9.8&35.02&32.67&11.41&11.43&39.96&38.11&13.04&13.07&44.82&43.56&14.68&14.7&49.65&49\\ \cline{2-30}
&8&5.83&5.6&21.95&18.67&7.45&7.47&27.84&24.89&9.2&9.33&33.77&31.11&11.03&11.2&39.61&37.33&12.91&13.07&45.35&43.56&14.81&14.93&51.01&49.78&16.74&16.8&56.63&56\\ \cline{2-30}
&9&6.27&6.3&23.95&21&8.13&8.4&30.74&28&10.15&10.5&37.54&35&12.26&12.6&44.2&42&14.41&14.7&50.74&49&16.59&16.8&57.2&56&18.79&18.9&63.62&63\\ \cline{2-30}
&10&6.72&7&25.96&23.33&8.82&9.33&33.66&31.11&11.11&11.67&41.31&38.89&13.48&14&48.79&46.67&15.91&16.33&56.13&54.44&18.37&18.67&63.39&62.22&20.85&21&70.6&70\\ \hline
\multirow{6}{*}{0.8}& 5&7.95&6&42.4&30&9.55&8&50.73&40&11.22&10&59.09&50&12.94&12&67.38&60&14.69&14&75.61&70&16.45&16&83.77&80&18.22&18&91.9&90\\ \cline{2-30}
&6&8.86&7.2&47.98&36&10.88&9.6&58.48&48&12.99&12&68.95&60&15.15&14.4&79.31&72&17.35&16.8&89.56&84&19.56&19.2&99.75&96&21.78&21.6&109.89&108\\ \cline{2-30}
&7&9.77&8.4&53.59&42&12.21&11.2&66.27&56&14.76&14&78.84&70&17.37&16.8&91.24&84&20.01&19.6&103.51&98&22.66&22.4&115.71&112&25.33&25.2&127.87&126\\ \cline{2-30}
&8&10.67&9.6&59.23&48&13.54&12.8&74.08&64&16.53&16&88.73&80&19.58&19.2&103.17&96&22.67&22.4&117.46&112&25.77&25.6&131.68&128&28.88&28.8&145.85&144\\ \cline{2-30}
&9&11.57&10.8&64.89&54&14.87&14.4&81.91&72&18.3&18&98.63&90&21.8&21.6&115.09&108&25.33&25.2&131.41&126&28.88&28.8&147.65&144&32.44&32.4&163.84&162\\ \cline{2-30}
&10&12.47&12&70.59&60&16.2&16&89.75&80&20.08&20&108.52&100&24.02&24&127.01&120&27.99&28&145.35&140&31.99&32&163.61&160&35.99&36&181.82&180\\ \hline
 \multirow{6}{*}{0.9}&5&18.73&13.5&195.64&135&22.38&18&232.42&180&26.11&22.5&269&225&29.87&27&305.37&270&33.64&31.5&341.61&315&37.42&36&377.78&360&41.21&40.5&413.9&405\\ \cline{2-30}
&6&21.09&16.2&222.25&162&25.69&21.6&268.27&216&30.36&27&313.92&270&35.07&32.4&359.32&324&39.79&37.8&404.57&378&44.53&43.2&449.75&432&49.26&48.6&494.89&486\\ \cline{2-30}
&7&23.45&18.9&248.93&189&28.99&25.2&304.14&252&34.62&31.5&358.84&315&40.28&37.8&413.25&378&45.95&44.1&467.52&441&51.63&50.4&521.72&504&57.31&56.7&575.87&567\\ \cline{2-30}
&8&25.81&21.6&275.68&216&32.3&28.8&340.02&288&38.88&36&403.75&360&45.49&43.2&467.18&432&52.11&50.4&530.47&504&58.73&57.6&593.68&576&65.37&64.8&656.85&648\\ \cline{2-30}
&9&28.17&24.3&302.46&243&35.62&32.4&375.9&324&43.14&40.5&448.65&405&50.69&48.6&521.1&486&58.26&56.7&593.41&567&65.84&64.8&665.64&648&73.42&72.9&737.84&729\\ \cline{2-30}
&10&30.54&27&329.25&270&38.93&36&411.77&360&47.4&45&493.54&450&55.9&54&575.02&540&64.42&63&656.35&630&72.94&72&737.61&720&81.47&81&818.82&810\\ \hline
\end{tabular}}
\end{center}
\end{table}
\end{landscape}\newpage
\begin{landscape}
 \begin {table}[H]
\footnotesize
\caption{Exact and approximate closed-form of mean and variance for various NGNB parameters ($\gamma$,q,k), $\gamma$ (1.2, 1.4, 1.5, 1.6, 1.8, 2), q=(0.2, 0.4, 0.5, 0.6, 0.7, 0.8), k=(5, 6, 7, 8, 9, 10) \label{tab: COMNB moments3}}
\begin{center}
\scalebox{0.70}{
\begin{tabular}{|c|c|c|c|c|c|c|c|c|c|c|c|c|c|c|c|c|c|c|c|c|c|c|c|c|c|} \hline
  \multirow{1}{*}{$q$} & \multicolumn{5}{|c|}{$\gamma=1.2$}&\multicolumn{4}{|c|}{$\gamma=1.4$}&\multicolumn{4}{|c|}{$\gamma=1.5$}&\multicolumn{4}{|c|}{$\gamma=1.6$}&\multicolumn{4}{|c|}{$\gamma=1.8$} &\multicolumn{4}{|c|}{$\gamma=2$}   \\ \cline{2-26}
  &k& E(Y)&$\frac{k \gamma q}{1-q}$ & Var(Y)&$\frac{k \gamma q}{(1-q)^2}$&E(Y)&$\frac{k \gamma q}{1-q}$ & Var(Y)&$\frac{k \gamma q}{(1-q)^2}$&E(Y)&$\frac{k \gamma q}{1-q}$ & Var(Y)&$\frac{k \gamma q}{(1-q)^2}$&E(Y)&$\frac{k \gamma q}{1-q}$ & Var(Y)&$\frac{k \gamma q}{(1-q)^2}$&E(Y)&$\frac{k \gamma q}{1-q}$ & Var(Y)&$\frac{k \gamma q}{(1-q)^2}$&E(Y)&$\frac{k \gamma q}{1-q}$ & Var(Y)&$\frac{k \gamma q}{(1-q)^2}$\\ \hline
   \multirow{6}{*}{0.2}&5& 1.61& 1.50&  1.93&  1.88& 2.02& 1.75&  2.30&  2.19& 2.23& 1.88&  2.48&  2.34& 2.46& 2.00&  2.65&  2.50& 2.91& 2.25&  2.99&  2.81& 3.37& 2.50&  3.31&  3.12\\ \cline{2-26}
 &6& 1.97& 1.80&  2.34&  2.25& 2.49& 2.10&  2.79&  2.62& 2.76& 2.25&  3.00&  2.81& 3.04& 2.40&  3.22&  3.00& 3.61& 2.70&  3.63&  3.37& 4.19& 3.00&  4.04&  3.75\\ \cline{2-26}
 &7& 2.32& 2.10&  2.74&  2.62& 2.95& 2.45&  3.27&  3.06& 3.28& 2.62&  3.53&  3.28& 3.62& 2.80&  3.78&  3.50& 4.30& 3.15&  4.28&  3.94& 5.00& 3.50&  4.77&  4.38\\ \cline{2-26}
 &8& 2.68& 2.40&  3.14&  3.00& 3.42& 2.80&  3.75&  3.50& 3.81& 3.00&  4.05&  3.75& 4.20& 3.20&  4.34&  4.00& 5.00& 3.60&  4.92&  4.50& 5.81& 4.00&  5.50&  5.00\\ \cline{2-26}
 &9& 3.04& 2.70&  3.54&  3.37& 3.89& 3.15&  4.24&  3.94& 4.33& 3.38&  4.58&  4.22& 4.78& 3.60&  4.91&  4.50& 5.69& 4.05&  5.57&  5.06& 6.62& 4.50&  6.23&  5.62\\ \cline{2-26}
&10& 3.39& 3.00&  3.94&  3.75& 4.35& 3.50&  4.72&  4.38& 4.85& 3.75&  5.10&  4.69& 5.35& 4.00&  5.48&  5.00& 6.38& 4.50&  6.22&  5.62& 7.43& 5.00&  6.96&  6.25\\ \hline
 \multirow{6}{*}{0.4}&5& 4.13& 4.00&  6.59&  6.67& 4.95& 4.67&  7.60&  7.78& 5.36& 5.00&  8.10&  8.33& 5.79& 5.33&  8.59&  8.89& 6.63& 6.00&  9.57& 10.00& 7.48& 6.67& 10.55& 11.11\\ \cline{2-26}
 &6& 5.00& 4.80&  7.95&  8.00& 6.03& 5.60&  9.21&  9.33& 6.56& 6.00&  9.83& 10.00& 7.08& 6.40& 10.44& 10.67& 8.14& 7.20& 11.67& 12.00& 9.21& 8.00& 12.88& 13.33\\ \cline{2-26}
 &7& 5.87& 5.60&  9.31&  9.33& 7.12& 6.53& 10.81& 10.89& 7.74& 7.00& 11.56& 11.67& 8.38& 7.47& 12.29& 12.44& 9.65& 8.40& 13.76& 14.00&10.93& 9.33& 15.22& 15.56\\ \cline{2-26}
 &8& 6.75& 6.40& 10.67& 10.67& 8.20& 7.47& 12.42& 12.44& 8.93& 8.00& 13.29& 13.33& 9.67& 8.53& 14.15& 14.22&11.16& 9.60& 15.86& 16.00&12.65&10.67& 17.56& 17.78\\ \cline{2-26}
& 9& 7.62& 7.20& 12.03& 12.00& 9.28& 8.40& 14.03& 14.00&10.12& 9.00& 15.02& 15.00&10.97& 9.60& 16.00& 16.00&12.66&10.80& 17.96& 18.00&14.37&12.00& 19.90& 20.00\\ \cline{2-26}
&10& 8.49& 8.00& 13.40& 13.33&10.36& 9.33& 15.64& 15.56&11.31&10.00& 16.75& 16.67&12.26&10.67& 17.85& 17.78&14.17&12.00& 20.05& 20.00&16.09&13.33& 22.24& 22.22\\ \hline
 \multirow{6}{*}{0.5}&5& 6.09& 6.00& 11.75& 12.00& 7.21& 7.00& 13.47& 14.00& 7.77& 7.50& 14.33& 15.00& 8.34& 8.00& 15.18& 16.00& 9.47& 9.00& 16.87& 18.00&10.61&10.00& 18.56& 20.00\\ \cline{2-26}
 &6& 7.37& 7.20& 14.18& 14.40& 8.77& 8.40& 16.32& 16.80& 9.47& 9.00& 17.39& 18.00&10.18& 9.60& 18.45& 19.20&11.60&10.80& 20.57& 21.60&13.02&12.00& 22.68& 24.00\\ \cline{2-26}
 &7& 8.65& 8.40& 16.61& 16.80&10.33& 9.80& 19.18& 19.60&11.18&10.50& 20.45& 21.00&12.03&11.20& 21.73& 22.40&13.73&12.60& 24.27& 25.20&15.44&14.00& 26.79& 28.00\\ \cline{2-26}
 &8& 9.93& 9.60& 19.04& 19.20&11.89&11.20& 22.03& 22.40&12.88&12.00& 23.52& 24.00&13.87&12.80& 25.01& 25.60&15.86&14.40& 27.97& 28.80&17.85&16.00& 30.92& 32.00\\ \cline{2-26}
 &9&11.21&10.80& 21.46& 21.60&13.46&12.60& 24.89& 25.20&14.58&13.50& 26.59& 27.00&15.72&14.40& 28.28& 28.80&17.99&16.20& 31.66& 32.40&20.27&18.00& 35.04& 36.00\\ \cline{2-26}
&10&12.49&12.00& 23.89& 24.00&15.02&14.00& 27.74& 28.00&16.29&15.00& 29.65& 30.00&17.56&16.00& 31.56& 32.00&20.12&18.00& 35.36& 36.00&22.68&20.00& 39.16& 40.00\\ \hline
\multirow{6}{*}{0.6}& 5& 9.02& 9.00& 21.90& 22.50&10.56&10.50& 25.03& 26.25&11.33&11.25& 26.58& 28.12&12.10&12.00& 28.13& 30.00&13.66&13.50& 31.22& 33.75&15.21&15.00& 34.31& 37.50\\ \cline{2-26}
 &6&10.91&10.80& 26.43& 27.00&12.83&12.60& 30.33& 31.50&13.80&13.50& 32.27& 33.75&14.76&14.40& 34.21& 36.00&16.70&16.20& 38.07& 40.50&18.65&18.00& 41.93& 45.00\\ \cline{2-26}
 &7&12.79&12.60& 30.96& 31.50&15.10&14.70& 35.63& 36.75&16.26&15.75& 37.96& 39.37&17.42&16.80& 40.29& 42.00&19.75&18.90& 44.92& 47.25&22.08&21.00& 49.56& 52.50\\ \cline{2-26}
 &8&14.68&14.40& 35.49& 36.00&17.37&16.80& 40.94& 42.00&18.73&18.00& 43.65& 45.00&20.08&19.20& 46.36& 48.00&22.80&21.60& 51.78& 54.00&25.52&24.00& 57.18& 60.00\\ \cline{2-26}
 &9&16.56&16.20& 40.02& 40.50&19.65&18.90& 46.24& 47.25&21.19&20.25& 49.35& 50.62&22.74&21.60& 52.44& 54.00&25.85&24.30& 58.63& 60.75&28.96&27.00& 64.80& 67.50\\ \cline{2-26}
&10&18.45&18.00& 44.55& 45.00&21.92&21.00& 51.55& 52.50&23.66&22.50& 55.04& 56.25&25.40&24.00& 58.52& 60.00&28.89&27.00& 65.48& 67.50&32.39&30.00& 72.42& 75.00\\ \hline
\multirow{6}{*}{0.7}& 5&13.88&14.00& 45.27& 46.67&16.10&16.33& 51.61& 54.44&17.21&17.50& 54.78& 58.33&18.33&18.67& 57.94& 62.22&20.56&21.00& 64.26& 70.00&22.80&23.33& 70.57& 77.78\\ \cline{2-26}
 &6&16.77&16.80& 54.63& 56.00&19.55&19.60& 62.55& 65.33&20.94&21.00& 66.51& 70.00&22.33&22.40& 70.46& 74.67&25.13&25.20& 78.36& 84.00&27.92&28.00& 86.24& 93.33\\ \cline{2-26}
 &7&19.66&19.60& 63.99& 65.33&22.99&22.87& 73.50& 76.22&24.67&24.50& 78.24& 81.67&26.34&26.13& 82.98& 87.11&29.69&29.40& 92.46& 98.00&33.04&32.67&101.92&108.89\\ \cline{2-26}
 &8&22.55&22.40& 73.35& 74.67&26.44&26.13& 84.44& 87.11&28.39&28.00& 89.98& 93.33&30.34&29.87& 95.51& 99.56&34.25&33.60&106.56&112.00&38.16&37.33&117.60&124.44\\ \cline{2-26}
 &9&25.44&25.20& 82.71& 84.00&29.89&29.40& 95.39& 98.00&32.12&31.50&101.71&105.00&34.35&33.60&108.03&112.00&38.82&37.80&120.66&126.00&43.29&42.00&133.28&140.00\\ \cline{2-26}
&10&28.33&28.00& 92.08& 93.33&33.33&32.67&106.33&108.89&35.84&35.00&113.45&116.67&38.35&37.33&120.56&124.44&43.38&42.00&134.76&140.00&48.41&46.67&148.96&155.56\\ \hline
\multirow{6}{*}{0.8}&5&23.57&24.00&116.15&120.00&27.14&28.00&132.28&140.00&28.93&30.00&140.33&150.00&30.71&32.00&148.38&160.00&34.29&36.00&164.48&180.00&37.87&40.00&180.57&200.00\\ \cline{2-26}
 &6&28.46&28.80&140.18&144.00&32.93&33.60&160.33&168.00&35.16&36.00&170.40&180.00&37.40&38.40&180.46&192.00&41.87&43.20&200.58&216.00&46.35&48.00&220.69&240.00\\ \cline{2-26}
 &7&33.35&33.60&164.21&168.00&38.71&39.20&188.39&196.00&41.40&42.00&200.47&210.00&44.08&44.80&212.54&224.00&49.45&50.40&236.68&252.00&54.82&56.00&260.82&280.00\\ \cline{2-26}
 &8&38.25&38.40&188.24&192.00&44.50&44.80&216.44&224.00&47.63&48.00&230.54&240.00&50.76&51.20&244.62&256.00&57.02&57.60&272.79&288.00&63.29&64.00&300.94&320.00\\ \cline{2-26}
 &9&43.14&43.20&212.27&216.00&50.29&50.40&244.50&252.00&53.86&54.00&260.60&270.00&57.44&57.60&276.70&288.00&64.60&64.80&308.89&324.00&71.76&72.00&341.06&360.00\\ \cline{2-26}
&10&48.03&48.00&236.30&240.00&56.08&56.00&272.56&280.00&60.10&60.00&290.67&300.00&64.12&64.00&308.78&320.00&72.18&72.00&344.99&360.00&80.23&80.00&381.19&400.00\\ \hline
\end{tabular}}
\end{center}
\end{table}
\end{landscape}

  \bibliography{}
 
\end{document}